\documentclass[11pt,twoside]{amsart}
\usepackage{amsmath, amsthm, amscd, amsfonts, amssymb, graphicx, color}
\usepackage[bookmarksnumbered, colorlinks, plainpages]{hyperref}

\newtheorem{thm}{Theorem}[section]

\newtheorem{lem}[thm]{Lemma}
\newtheorem{prop}[thm]{Proposition}
\newtheorem{defn}[thm]{Definition}

\numberwithin{equation}{section}
\def\pn{\par\noindent}

\newcommand{\be}{\begin{equation}}
\newcommand{\ee}{\end{equation}}
\newcommand{\ben}{\begin{enumerate}}
\newcommand{\een}{\end{enumerate}}
\newcommand{\beq}{\begin{eqnarray}}
\newcommand{\eeq}{\end{eqnarray}}
\newcommand{\beqn}{\begin{eqnarray*}}
\newcommand{\eeqn}{\end{eqnarray*}}

\newcommand{\pa}{\partial}
\newcommand{\bpf}{\begin{proof}}
\newcommand{\epf}{\end{proof}}
\newcommand{\bl}{\begin{lem}}
\newcommand{\el}{\end{lem}}
\newcommand{\bp}{\begin{prop}}
\newcommand{\ep}{\end{prop}}
\newcommand{\bd}{\begin{defn}}
\newcommand{\ed}{\end{defn}}
\newcommand{\bt}{\begin{thm}}
\newcommand{\et}{\end{thm}}
\def\nn{\nonumber}
\begin{document}

\leftline{ \scriptsize \it Bulletin of the Iranian Mathematical
Society  Vol. {\bf\rm XX} No. X {\rm(}201X{\rm)}, pp XX-XX.}

\vspace{1.3 cm}

\title{On quasi-Eienstein Finsler spaces }
\author{Behroz Bidabad$^*$ and Mohamad Yarahmadi}

\thanks{{\scriptsize
\hskip -0.4 true cm MSC(2010): Primary: 53C60; Secondary: 53C44,
35C08.
$*$Corresponding author
\newline\indent{\scriptsize $\copyright$ 2012 Iranian Mathematical
Society}}}

\maketitle

\begin{center}
\end{center}

\begin{abstract}
 The notion of quasi-Einstein metric in physics is equivalent to the notion of Ricci soliton in Riemannian spaces. Quasi-Einstein metrics serve also as solution to the Ricci flow equation. Here, the Riemannian metric is replaced by a Hessian matrix derived from a Finsler structure and a quasi-Einstein Finsler metric is defined. In  compact case, it is proved that the quasi-Einstein metrics are solution to the Finslerian Ricci flow and conversely, certain form of solutions to the Finslerian Ricci flow are quasi-Einstein  Finsler metrics.
\end{abstract}

\vskip 0.2 true cm


\pagestyle{myheadings}
\markboth{\rightline {\scriptsize  Bidabad and Yarahmadi}}
         {\leftline{\scriptsize On quasi-Eienstein Finsler spaces}}

\bigskip
\bigskip


\section{\bf Introduction}
\vskip 0.4 true cm
The Ricci flow in Riemannian geometry was introduced by R. S. Hamilton in 1982, cf. \cite{Ha}, and since then
has been extensively studied thanks to its applications in geometry, physics and different branches of real world problems. Theoretical physicists have also been looking into the equation of quasi-Einstein metrics in relation with string theory.
Ricci flow is a process that deforms the metric of a Riemannian manifold in a way formally analogous to the diffusion of heat.
G. Perelman used Ricci flow to prove the  Poincar\'{e} conjecture.
Quasi-Einstein metrics or Ricci solitons are considered as a solution to the Ricci flow equation and are subject of great interest in geometry and physics.
 Let $(M,g)$ be  a Riemannian manifold, a triple $(M,g,X)$ is said to be a \emph{quasi-Einstein metric} or \emph{Ricci soliton} if $g$ satisfies  the  equation
 \begin{align}\label{Eq;RST}
{2Rc+\mathcal{L}_{X} g=2\lambda g ,}
\end{align}
where $Rc$ is the Ricci tensor, $X$  a smooth vector field on $M$, $\mathcal{L}_{X}$  the Lie derivative along $X$ and $\lambda$  a real constant.
 If the vector field $X$ is gradient of a potential function $f$, then $(M,g,X)$ is said to be \emph{ gradient} and  (\ref{Eq;RST}) takes the familiar form
\begin{equation}
{Rc+\nabla\nabla f=\lambda g.}
\end{equation}
Perelman has proved that  on a compact Riemannian  manifold  every  Ricci soliton is  gradient, cf.  \cite{pre}.
Moreover, on a compact Riemannian manifold a quasi-Einstein metric is a special solution to the Ricci flow equation defined by
\be\label{E;Ricci flow}
\frac{\partial}{\partial t} g(t)=-2Rc, \quad {g(t=0):=g_{_0}}.
\ee
A quasi-Einstein metric is considered as special solution to the Ricci flow  in Riemannian geometry.

Some recent work has focused on the natural question of extending this notion to the Finsler geometry as a natural generalization of Riemannian geometry.

In this work, the Akbar-Zadeh's Ricci tensor is used to  determine  the notion of quasi-Einstein  Finsler metric and it is shown that in the case of compact manifolds, it is a solution to the Ricci flow equation used by D. Bao cf. \cite{bao} in Finsler geometry and vice versa. More precisely,  it is proved that if there is a quasi-Einstein  Finsler metric on a compact manifold then there exists a solution to the Ricci flow equation. Conversely, certain form of solutions to the Ricci flow are quasi-Einstein  Finsler metrics.

\section{\bf {\bf \em{\bf Preliminaries and notations}}}
\vskip 0.4 true cm

Let $M$ be a real n-dimensional manifold differentiable. We denote by $TM$ its tangent bundle and by $\pi :TM_{0} \longrightarrow M$,  fiber bundle of non zero tangent vectors.
A \emph{Finsler structure} on $M$ is a function $F:TM\longrightarrow [0,\infty)$, with the following properties:\\
I. Regularity: $F$ is $C^{\infty}$ on the entire slit tangent bundle $TM_{0}=TM\backslash 0$.\\
II. Positive homogeneity: $F(x,\lambda y)=\lambda F(x,y)$ for all $\lambda > 0$.\\
III. Strong convexity: The $n\times n$ Hessian matrix $(g_{ij})_F=([\frac{1}{2} F^{2}]_{y^{i} y^{j}})$ is positive definite at every point of $TM_{0}$.
A \emph{Finsler manifold} $(M,F)$ is a pair consisting of a differentiable manifold $M$ and a Finsler structure $F$.
The formal Christoffel symbols of  second kind and spray coefficients are respectively denoted here by
\begin{align}\label{chri}
(\gamma^i_{jk})_{_F}:=g^{is}\frac{1}{2}\big(\frac{\partial g_{sj}}{\partial x^k}-\frac{\partial g_{jk}}{\partial x^s}+\frac{\partial g_{ks}}{\partial x^j}\big),
\end{align}
 where $g_{ij}(x,y)=[\frac{1}{2} F^2]_{y^{i}y^{j}}$, and
 \begin{align}\label{E,G}
 G^i_{_F}&:=\frac{1}{2}(\gamma^i_{jk})_{_F}y^j y^k.
\end{align}
We consider also the \emph{reduced curvature tensor} $R^i_k$ which is expressed entirely in terms of the $x$ and $y$ derivatives of spray coefficients $G^i_{_F}$.
 \begin{align}\label{E,Ricci scalar}
 (R^{i}_{k})_{_F}:=\frac{1}{F^2}\big(2\frac{\partial G^i}{\partial x^k}-\frac{\partial^2 G^i}{\partial x^j \partial y^k}y^j +2G^j\frac{\partial^2 G^i}{\partial y^j \partial y^k} - \frac{\partial G^i}{\partial y^j}\frac{\partial G^j}{\partial y^k}\big).
\end{align}
In the general Finslerian setting, one of the remarkable definitions of Ricci tensors is  introduced by H. Akbar-Zadeh \cite{AZ} as follows.
\begin{equation*}
{Ric_{jk} :=[\frac{1}{2} F^{2} Ric]_{y^{j} y^{k}}},
\end{equation*}
where
$
Ric=R^{i}_{i}
$
and  $R^i_k$ is defined by (\ref{E,Ricci scalar}). One of the advantages of the Ricci  quantity defined here  is its independence to the choice of the Cartan, Berwald or Chern(Rund) connections. Based on the Akbar-Zadeh's Ricci tensor, in analogy with the equation (\ref{E;Ricci flow})
, D. Bao has considered, the following natural extension of \emph{Ricci flow } in Finsler geometry, cf.  \cite{bao},
\begin{align*}
\frac{\partial}{\partial t} g_{jk}=-2Ric_{jk}, \quad {g(t=0):=g_{_0}}.
\end{align*}
This equation is equivalent to the following differential equation
\begin{align}\label{E;Ricci}
{\frac{\partial}{\partial t}(\log F(t))=-Ric,}\quad {F(t=0):=F_{_0}},
\end{align}
where, $F_{_0}$ is the initial Finsler structure.
Let $V=v^{i}(x) \frac{\partial}{\partial x_i}$  be a vector field on $M$.  If \{$\varphi_t$\} is the local one-parameter group of $M$ generated by $V$, then it induces an infinitesimal point transformation on $M$ defined by $\varphi^\star_t (x^{i}):= \bar{x}^{i}$, where $\bar{x}^i=x^{i}+v^{i}(x) t$. This is naturally extended to the point transformation $\tilde{\varphi_t}$ on the tangent bundle $TM$ defined by $\tilde{\varphi_t}^{\star} := (\bar{x}^{i},\bar{y}^{i})$, where
\begin{align} \label{E;T}
\bar{x}^{i}=x^{i}+v^{i}(x) t, \quad \bar{y}^{i}=y^{i}+\frac{\partial v^{i}}{\partial x^{m}}  y^{m} t.
\end{align}
It can be shown that, $\{\tilde{\varphi_t}\}$ induces a vector field $\hat{V}=v^{i}(x) \frac{\partial}{\partial x^{i}}+y^{j} \frac{\partial v^{i}}{\partial x^{j}}  \frac{\partial}{\partial y^{i}}$ on $TM$ called the complete lift of $V$. The
 one-parameter group associated to  the complete lift $\hat{V}$ is given by  $\tilde{\varphi_t}=y^i \frac{\partial \varphi_t}{\partial x^i}$.

The \emph{Lie derivative} of an arbitrary geometric object $\Upsilon^{I}(x,y)$ on $TM$, where $I$ is a mixed multi index, with respect to the complete lift $\hat{V}$ of a vector field $V$ on $M$, is defined by
\begin{align} \label{Df;LieDer}
{\mathcal{L}_{\hat{V}} \Upsilon^I=\lim_{t\rightarrow 0}\frac{\tilde{\varphi_t}^{\star}(\Upsilon^{I})-\Upsilon^{I}}{t}}=\frac{d}{dt}\tilde{\varphi_t}^{\star}(\Upsilon^{I}),
\end{align}
where $\tilde{\varphi_t}^\star(\Upsilon^{I})$ is the deformation of $\Upsilon^{I}(x,y)$ under the extended point transformation (\ref{E;T}). Whenever the geometric object $\Upsilon^{I}$ is a tensor field, $\tilde{\varphi_t}^\star(\Upsilon^{I})$ coincides with the classical notation of pullback of $\Upsilon^{I}(x,y)$, cf. \cite{Johari}.

\section{\bf {\bf \em{\bf Quasi-Einstein Finsler metrics}}}
\vskip 0.4 true cm

Let $(M,F_0)$ be a Finsler manifold and $V=v^{i}(x) \frac{\partial}{\partial x^{i}}$
 a vector field on $M$. We call the triple $(M,F_0,V)$ a Finslerian \emph{quasi-Einstein} or a \emph{Ricci soliton} if ${g_{jk}}$ the Hessian related to the Finsler structure $F_0$ satisfies
 \begin{align}\label{Eq;DefRicciSoliton+}
2Ric_{jk}+\mathcal{L}_{\hat{V}}  {g_{jk}}=2 \lambda g_{jk},
\end{align}
where,   $\hat V$ is complete lift of $V$ and  $\lambda \in {\mathbb{R}}$. A  moment's thought shows that this equation leads to
 \begin{align}\label{Eq;DefRicciSoliton}
2F_{0}^{2}Ric_{_{F_0}}+\mathcal{L}_{\hat{V}} { F_{0}^{2}}=2 \lambda F_{0}^{2}.
\end{align}


\begin{lem}
Let $M$ be a  differentiable manifold, $F_0$ a Finsler structure and $\varphi_t$  a family of diffeomorphisms on $M$.  Then the pull back of extended point transformation  $\tilde{\varphi_t}^\star(F_0): TM\longrightarrow [0,\infty)$ is also  a Finsler structure on $M$.
\end{lem}
\begin{proof}
If we put $\tilde{\varphi_t}^{\star}x^i=\tilde x^i$ and $\tilde{\varphi_t}^{\star}y^i=\tilde y^i$. We should show that $\tilde{\varphi_t}^{\star}(F_0):= F_0 \circ \tilde{\varphi_t}:TM_0\longrightarrow [0,\infty)$ satisfies  the three conditions I, II and  III in the definition of Finsler structure.
Clearly we have the regularity condition since  ${F_0}$ and $\tilde{\varphi_t}$ are  $C^{\infty}$ on $TM_{0}$, and  so  is ${F_{0}}\circ{\widetilde{\varphi_t}}$.
\begin{align*}
\tilde{\varphi_t}^{\star}F_0(x,\lambda y)&={F_{0}}(\tilde{\varphi_t}(x,\lambda y))={F_0}(\varphi_t (x),\lambda y^{j} \frac{\partial \varphi_t}{\partial x^{j}})\\
&=\lambda F_{0}(\varphi_t(x),y^{j} \frac{\partial \varphi_t}{\partial x^{j}})=
\lambda F_{0}(\tilde{\varphi_t}(x,y))=\lambda \tilde{\varphi_t}^{\star}(F_0)(x,y).
\end{align*}
Thus the positive homogeneity is satisfied. For
strong convexity we have
\begin{align*}
[\frac{1}{2} \left(\tilde{\varphi_t}^\star F_0\right)^2]_{\tilde y^{i} \tilde y^{j}}&=\frac{1}{2}\frac{\partial^2 \big( \left(\tilde{\varphi_t}^\star F_0 \right)^2 \big)}{\partial \tilde y^i \partial \tilde y^j}=\frac{1}{2}\frac{\partial^2 \big( (F_0\circ\tilde{\varphi_t})(F_0\circ\tilde{\varphi_t})\big)}{\partial \tilde y^i \partial \tilde y^j}\\
&=\frac{1}{2}\frac{\partial^2 \big(F_0^2\circ\tilde{\varphi_t}  \big)}{\partial \tilde y^i \partial \tilde y^j}
=\frac{1}{2} {\frac{\partial ^2 (\tilde{\varphi_t}^{\star}F_0^2)}{\partial \tilde y^{i} \partial \tilde y^{j}}}.
\end{align*}
One can easily check that
\begin{align*}
 {\frac{\partial  (\tilde{\varphi_t}^{\star}F_0^2)}{\partial \tilde y^{i}}}= \tilde{\varphi_t}^\star( \frac{\partial F_0^2}{\partial y^i }),
\end{align*}
from which
\begin{align*}
[\frac{1}{2} \left(\tilde{\varphi_t}^\star F_0\right)^2]_{\tilde y^{i} \tilde y^{j}}=\frac{1}{2} {\frac{\partial ^2 (\tilde{\varphi_t}^{\star}F_0^2)}{\partial \tilde y^{i} \partial \tilde y^{j}}}=\tilde{\varphi_t}^\star(\frac{1}{2} \frac{\partial^2 F_0^2}{\partial y^i \partial y^j})=\tilde{\varphi_t}^\star([\frac{1}{2} F_0^2]_{y^i y^j}).
\end{align*}
Using the facts that  $[\frac{1}{2} F_0^2]_{y^{i}y^{j}}$ is positive definite and $\tilde{\varphi_t}^{\star}$ is a diffeomorphism, $\tilde{\varphi_t}^{\star}([\frac{1}{2} F_0^2]_{y^{i}y^{j}})$ is also positive definite and  hence $\tilde{\varphi_t}^{\star}(F_0)$ satisfies   III. This completes the proof of Lemma 1.
\end{proof}
\begin{lem}
Let $\varphi_t$ be  a family of diffeomorphisms on  $M$ and $(\gamma^i_{jk})_{_{F_0}}$ and  $G^i_{_{F_0}}$ the Christoffel and spray coefficients related to the Finsler structure $F_0$, respectively.  Then we have
\begin{align}\label{E;ExtendedTransChris}
&\tilde{\varphi_t}^\star\big((\gamma^i_{jk})_{_{F_0}}\big)=
(\gamma^i_{jk})_{\tilde{\varphi_t}^\star(F_0)},\\\label{E;ExtendedTransSpray}
&\tilde{\varphi_t}^\star\big(G^i_{_{F_0}}\big)=G^i_{\tilde{\varphi_t}^\star(F_0)}.
\end{align}
\end{lem}
\begin{proof}
By definition, we have
\begin{align*}
\tilde{\varphi_t}^\star\big(\gamma^i_{jk})_{_{F_0}}&=
\tilde{\varphi_t}^\star\big(g^{is}\frac{1}{2}(\frac{\partial g_{sj}}{\partial x^k}-\frac{\partial g_{jk}}{\partial x^s}+\frac{\partial g_{ks}}{\partial x^j})\big)\\&=
\tilde{\varphi_t}^\star(g^{is})\tilde{\varphi_t}^\star(\frac{1}{2}(\frac{\partial g_{sj}}{\partial x^k}-\frac{\partial g_{jk}}{\partial x^s}+\frac{\partial g_{ks}}{\partial x^j}))\\
&=\tilde{\varphi_t}^\star(g^{is})\frac{1}{2}\big(\frac{\partial \tilde{\varphi_t}^\star(g_{sj})}{\partial \tilde x^k}-\frac{\partial \tilde{\varphi_t}^\star(g_{jk})}{\partial \tilde x^s}+\frac{\partial \tilde{\varphi_t}^\star(g_{ks})}{\partial \tilde x^j}\big)\\
&=(\gamma^i_{jk})_{\tilde{\varphi_t}^\star(F_0)}.
\end{align*}
Next, we have
\begin{align*}
\tilde{\varphi_t}^\star\big(G^i_{_{F_0}}\big)&=
\tilde{\varphi_t}^\star\big(\frac{1}{2}(\gamma^i_{jk})_{_{F_0}} y^j y^k\big)=
\frac{1}{2}\tilde{\varphi_t}^\star\big((\gamma^i_{jk})_{_{F_0}}\big) \tilde{\varphi_t}^\star y^j \tilde{\varphi_t}^\star y^k \\ &=\frac{1}{2}(\gamma^i_{jk})_{_{\tilde{\varphi_t}^\star(F_0)}}\tilde y^j \tilde y^k=G^i_{\tilde{\varphi_t}^\star(F_0)}.
\end{align*}
which completes the proof of Lemma 2.
\end{proof}
\begin{lem}
Let $\varphi_t$ be  a family of diffeomorphisms on  $M$ and $Ric_{_{F_0}}$ the Ricci scalar related to the Finsler structure $F_0$, then we have
 \begin{align}\label{E;1}
Ric_{_{\mu F_{_{0}}}}=\frac{1}{\mu^2} Ric_{_{F_{_0}}},
\end{align}
 for all $\mu>0$, and
\begin{align}\label{E;2}
\tilde{\varphi_t}^{\star}(Ric_{_{F_0}})=Ric_{_{\tilde{\varphi_t}^{\star}(F_0)}}.
\end{align}
\end{lem}
\begin{proof}
The spray coefficient $G^i_{_{F_0}}$ are two-homogeneous, that is  for all $\lambda>0$ we have $G^i_{_{F_0}}(x,\lambda y)=\lambda^2 G^i_{_{F_0}}(x, y)$. On the other hand   for all $\mu>0$,  $G^i_{_{\mu F_0}}= G^i_{_{F_0}}$.
In fact
\begin{align*}
(g_{ij})_{_{\mu F_0}}=\frac{1}{2}\frac{\partial^2\big((\mu F_0)^2\big)}{\partial y^i y^j}=\frac{1}{2}\frac{\partial^2 (\mu^2 F_0^2)}{\partial y^i y^j}=\frac{1}{2}\mu^2 \frac{\partial^2 ( F_0^2)}{\partial y^i y^j}=\mu^2 (g_{ij})_{_{F_0}},
\end{align*}
and hence
\begin{align*}
(\gamma^i_{jk})_{_{\mu F_0}}&=\frac{1}{2}(g^{is})_{_{\mu F_0}}\big(\frac{\partial((g_{sj})_{_{\mu F_0}})}{\partial x^k}-\frac{\partial((g_{jk})_{_{\mu F_0}})}{\partial x^s}+\frac{\partial((g_{ks})_{_{\mu F_0}})}{\partial x^j}\big)\\&=\frac{1}{2}\frac{1}{\mu^2}(g^{is})_{_{F_0}}\big(\frac{\partial(\mu^2(g_{sj})_{_{ F_0}})}{\partial x^k}-\frac{\partial(\mu^2(g_{jk})_{_{F_0}})}{\partial x^s}+\frac{\partial(\mu^2(g_{ks})_{_{F_0}})}{\partial x^j}\big)\\
&=\frac{1}{2}\frac{1}{\mu^2}(g^{is})_{_{F_0}}\mu^2\big(\frac{\partial((g_{sj})_{_{ F_0}})}{\partial x^k}-\frac{\partial((g_{jk})_{_{F_0}})}{\partial x^s}+\frac{\partial((g_{ks})_{_{F_0}})}{\partial x^j}\big)\\
&=\frac{1}{2}(g^{is})_{_{F_0}}\big(\frac{\partial((g_{sj})_{_{ F_0}})}{\partial x^k}-\frac{\partial((g_{jk})_{_{F_0}})}{\partial x^s}+\frac{\partial((g_{ks})_{_{F_0}})}{\partial x^j}\big)=(\gamma^i_{jk})_{_{F_0}},
\end{align*}
implies that
\begin{align}\label{Eq;Gmu}
G^i_{_{\mu F_0}}=\frac{1}{2}(\gamma^i_{jk})_{_{\mu F_0}} y^i y^j=\frac{1}{2}(\gamma^i_{jk})_{_{F_0}} y^i y^j=G^i_{_{F_0}}.
\end{align}
 By means of the definition of Ricci scalar  we obtain
\begin{align}\label{Eq;Riccimu}
(R^i_k)_{_{\mu F_0}}=\frac{1}{\mu^2 F_0^2}\big(2\frac{\partial G^i_{_{\mu F_0}}}{\partial x^k}-\frac{\partial^2G^i_{_{\mu F_0}}}{\partial x^j \partial y^k}y^j +2G^j\frac{\partial^2G^i_{_{\mu F_0}}}{\partial y^j \partial y^k} - \frac{\partial G^i_{_{\mu F_0}}}{\partial y^j}\frac{\partial G^j_{_{\mu F_0}}}{\partial y^k}\big).
\end{align}
Using (\ref{Eq;Gmu})   and   (\ref{Eq;Riccimu})   we get
 \begin{align*}
 (R^i_k)_{_{\mu F_0}}&=\frac{1}{\mu^2 F_0^2}\big(2\frac{\partial G^i_{_{F_0}}}{\partial x^k}-\frac{\partial^2 G^i_{_{F_0}}}{\partial x^j \partial y^k}y^j +2G^j\frac{\partial^2 G^i_{_{F_0}}}{\partial y^j \partial y^k} - \frac{\partial G^i_{_{F_0}}}{\partial y^j}\frac{\partial G^j_{_{F_0}}}{\partial y^k}\big)\\&=\frac{1}{\mu^2}(R^i_k)_{_ {F_0}}.
 \end{align*}
  Putting $i=k$ in this equation, we get $Ric_{_{\mu F_0}}=\frac{1}{\mu^2}Ric_{_{F_0}}$.
Therefore we have (\ref{E;1}).
Next, in order to prove (\ref{E;2}) we use (\ref{E;1}) as follows.
\begin{align*}
\tilde{\varphi_t}^{\star}((R^i_k)_{_ {F_0}})&=\tilde{\varphi_t}^{\star}\big(\frac{1}{F_0^2}(2\frac{\pa(G^i_{_{F_0}})}{\pa x^k}-\frac{\pa^2(G^i_{_{F_0}})}{\pa x^j \pa y^k}y^j\\
&+2G^j_{F_0}\frac{\pa^2(G^i_{_{F_0}})}{\pa y^j \pa y^k}-\frac{\pa(G^i_{_{F_0}})}{\pa y^j}\frac{\pa(G^i_{_{F_0}})}{\pa y^k})\big)\\
&=\tilde{\varphi_t}^{\star}(\frac{1}{F_0^2})\tilde{\varphi_t}^{\star}\big(2\frac{\pa(G^i_{_{F_0}})}{\pa x^k}-\frac{\pa^2(G^i_{_{F_0}})}{\pa x^j \pa y^k}y^j\\&+2G^j_{F_0}\frac{\pa^2(G^i_{_{F_0}})}{\pa y^j \pa y^k}-\frac{\pa(G^i_{_{F_0}})}{\pa y^j}\frac{\pa(G^i_{_{F_0}})}{\pa y^k}\big).
\end{align*}
Thus we get
\begin{align*}
\tilde{\varphi_t}^{\star}((R^i_k)_{_ {F_0}})=&\frac{1}{\tilde{\varphi_t}^{\star}(F_0^2)}\big(2\frac{\pa(\tilde{\varphi_t}^{\star}(G^i_{_{F_0}}))}{\pa \tilde x^k}-\frac{\pa^2(\tilde{\varphi_t}^{\star}(G^i_{_{F_0}}))}{\pa \tilde x^j \pa \tilde y^k} \tilde y^j\\ &+2\tilde{\varphi_t}^{\star}(G^j_{F_0})\frac{\pa^2(\tilde{\varphi_t}^{\star}(G^i_{_{F_0}}))}{\pa \tilde y^j \pa \tilde y^k}
-\frac{\pa(\tilde{\varphi_t}^{\star}(G^i_{_{F_0}}))}{\pa \tilde y^j}\frac{\pa(\tilde{\varphi_t}^{\star}(G^i_{_{F_0}}))}{\pa \tilde y^k}\big).
\end{align*}
Putting $i=k$ in this equation together with (\ref{E;ExtendedTransSpray}) implies
\begin{align*}
\tilde{\varphi_t}^{\star}(Ric_{_{F_0}})=Ric_{_{\tilde{\varphi_t}^{\star}(F_0)}}.
\end{align*}
This completes the proof of Lemma 3.
\end{proof}
Now, we are in a position to prove the  main result of this work. Let $M$ be a compact differentiable  manifold and $F_0$ an initial Finsler structure on $M$. If $(M,F_0,V)$ is a solution to the  (\ref{Eq;DefRicciSoliton}) then there is  a one-parameter family of scalars $\tau(t) $ and a one-parameter family of diffeomorphisms $\varphi_t$ on $M$ such that $(M,F(t))$ is a solution of the Ricci flow (\ref{E;Ricci}), where $F(t)$ is defined by
\begin{align}\label{E;F}
F^2(t)=\tau(t) \tilde{\varphi_t}^\star(F_0^2).
\end{align}
The converse of this assertion is also through, that is, if $(M,F(t))$ is a solution to the Finsler Ricci flow having the special form (\ref{E;F}), then there is a vector field $V$ on $M$ such that $(M,F_0,V)$ is quasi-Einstein.
In short we have the following theorem.
\begin{thm}
Let $(M,F_0)$ be a compact Finsler manifold and $(M,F_0,V)$  a quasi-Einstein space. Then there exists  a solution $(M,F(t))$ in the  form (\ref{E;F}) to the Finslerian Ricci flow.
 Conversely, if $(M,F(t))$ is a solution to the Finslerian Ricci flow having the  form (\ref{E;F}), then there is a vector field $V$ on $M$ such that $(M,F_0,V)$ is quasi-Einstein.
\end{thm}
\begin{proof}
Suppose that $(M,F_0,V)$ satisfies (\ref{Eq;DefRicciSoliton}) and consider a family of scalars $\tau(t)$ defined by  $$\tau(t):= 1-2\lambda t >0,$$ where $\lambda$ is a constant.
Next define a one-parameter family of vector fields $X_{_t}$ on $M$ by
$$X_{_{t}}(x):= \frac{1}{\tau(t)} V(x).$$
 Let $\varphi_t$ denote the diffeomorphisms generated by the family $X_{_t}$, where $\varphi_{_0}=Id_M$, and define a smooth one-parameter family of Finsler structures on $M$ by $$F^2(t)=\tau(t) \tilde{\varphi_t}^\star(F_0^2).$$
Thus we have
$$\log(F(t))=\frac{1}{2} \log(\tau(t) \tilde{\varphi_t}^\star(F_0^2)).$$
Using (\ref{Df;LieDer}) we have
  $\frac{d}{dt}\tilde{\varphi_t}^{\star}( F_0^2)=\tilde{\varphi_t}^\star(\mathcal{L}_{\hat{X_{_t}}} F_0^2)$
    from which we get
\begin{align*}
\frac{\partial }{\partial t}\log(F(t))&=\frac{1}{2} \frac{\tau^\prime(t) \tilde{\varphi_t}^\star(F_0^2)+\tau(t) \tilde{\varphi_t}^\star(\mathcal{L}_{\hat{X_{_t}}} F_0^2)}{\tau(t)\tilde{\varphi_t}^\star(F_0^2)}
\\&=\frac{1}{2} \frac{\tilde{\varphi_t}^\star(-2\lambda F_0^2+\mathcal{L}_{\hat{V}} F_0^2)}{\tau(t)\tilde{\varphi_t}^\star(F_0^2)}\\
&=\frac{1}{2} \frac{\tilde{\varphi_t}^\star(-2F_0^2 Ric_{_{F_0}})}{\tau(t)\tilde{\varphi_t}^\star(F_0^2)}
\\&=\frac{1}{2} \frac{(-2)\tilde{\varphi_t}^\star(F_0^2) \tilde{\varphi_t}^\star(Ric_{_{F_0}}) }{\tau(t)\tilde{\varphi_t}^\star(F_0^2)}\\
&=\frac{-Ric_{_{\tilde{\varphi_t}^\star(F_0)}}}{\tau(t)}
=-Ric_{(\tau(t))^{\frac{_1}{_2}} \tilde{\varphi_t}^\star(F_0)}\\&=-Ric_{_{F(t)}}.
\end{align*}
Therefore, $F(t)$ is a solution to the Ricci flow (\ref{E;Ricci}). Conversely, suppose that $(M,F(t))$ is a solution to the Ricci flow equation (\ref{E;Ricci}) having the form (\ref{E;F}).
We may assume without loss of generality that $\tau(0)=1$ and $\varphi_0=Id_M$, then we have
\begin{align}\label{Eq;RicFzero}
-Ric_{_{F_0}}&=\frac{\partial}{\partial t} (\log F(t))\vert_{t=0}\nn \\&=
\frac{\partial}{\partial t}(\frac{1}{2} \log(\tau(t)\tilde{\varphi_t}^\star F_0^2))\vert_{t=0}
\nn \\&=\frac{1}{2} \frac{\tau^\prime(t) \tilde{\varphi_t}^\star F_0^2+\tau(t) \tilde{\varphi_t}^\star(\mathcal{L}_{\widehat{X}(t)} F_0^2)}{\tau(t)\tilde{\varphi_t}^\star(F_0^2)}\vert_{t=0}\nn \\&=\frac{1}{2} \frac{\tau^\prime(0) F_0^2+\mathcal{L}_{\widehat{X}(0)} F_0^2}{F_0^2},
\end{align}
where $X(t)$ is the family of vector fields generating the diffeomorphism $\varphi_t$.
Thus, (\ref{Eq;RicFzero}) implies
 $$-2F_0^2 Ric_{_{F_0}}=\tau^\prime(0) F_0^2+\mathcal{L}_{\widehat{X}(0)} F_0^2.$$
Replacing $\lambda=\frac{-1}{2} \tau^\prime(0)$ and $V=X(0)$ we have the equation (\ref{Eq;DefRicciSoliton}).
This completes the proof of Theorem.
\end{proof}

\vskip 1.4 true cm

\begin{center}{\textbf{Acknowledgments}}
\end{center}
The first  author was supported in part by an INSF grant (89000676). \\ \\
\vskip 0.4 true cm



\bigskip
\bigskip


{\footnotesize \pn{\bf Behroz Bidabad }\; \\ {Faculty of
Mathematics and Computer Science}, {Amirkabir University of Technology, 15914, } {Tehran, Iran.}\\
{\tt Email: bidabad@aut.ac.ir}\\

{\footnotesize \pn{\bf  Mohamad Yarahmadi}\; \\ {Faculty of
Mathematics and Computer Science}, {Amirkabir University of Technology, 15914, } {Tehran, Iran.}\\
{\tt Email: m.yarahmadi@aut.ac.ir}\\
\end{document}